%% file: paper.tex
\documentclass[12pt]{elsarticle} 

\usepackage{algorithm}
\usepackage{amsmath}
\usepackage{amsthm}
\usepackage{amsfonts}
\usepackage{bm} 

\usepackage[colorlinks=true]{hyperref}
\usepackage[margin=2.25cm]{geometry}
\usepackage{cleveref}
\input{my-notation.tex}

\def\institution{}
\def\streetaddress{}
\def\city{}
\def\country{}
\def\postcode{}

\journal{ }

\begin{document}

\begin{frontmatter}

\title{Tensor decomposition for learning Gaussian mixtures from moments}

\author{Rima Khouja, Pierre-Alexandre Mattei, Bernard Mourrain}


\address{%
  \institution{Inria d'Universit\'e C\^ote d'Azur},
  \streetaddress{2004 route des Lucioles, B.P. 93},
  \postcode{06902}
  \city{Sophia Antipolis},
  \country{France}

}




\begin{abstract}
  In data processing and machine learning, an important challenge is
  to recover and exploit models that can represent accurately the
  data. We consider the problem of recovering Gaussian mixture models
  from datasets. We investigate symmetric tensor decomposition methods
  for tackling this problem, where the tensor is built from empirical
  moments of the data distribution. We consider identifiable tensors,
  which have a unique decomposition, showing that moment tensors built
  from spherical Gaussian mixtures have this property.  We prove that
  symmetric tensors with interpolation degree strictly less than half
  their order are identifiable and we present an algorithm, based on
  simple linear algebra operations, to compute their decomposition.
  Illustrative experimentations show the impact of the tensor
  decomposition method for recovering Gaussian mixtures, in comparison
  with other state-of-the-art approaches.
\end{abstract}

\end{frontmatter}

\section{Introduction}

With the relatively recent evolutions of information systems over the last decades, many
observations, measurements, data are nowadays available on a variety
of subjects. However, too much information can kill the information and
one of the main challenges remains to analyse and to model these data,
in order to recover and exploit hidden structures.

To tackle this challenge, popular Machine Learning technologies have
been developed and used successfully in several application domains
(e.g. in image recognition \cite{he2016deep}).
These techniques can be grouped in two main classes:
Supervised machine learning techniques are
approximating a model by optimising the parameters of an enough
general model (e.g. a Convolution Neural Network) from training data.
Unsupervised machine learning techniques are deducing the parameters
characterising a model directly from the given data, using an apriori
knowledge on the model.  The supervised approach requires
annotated data, with a training step that can introduce some bias in
the learned model.
The unsupervised approach can be applied
directly on a given data set avoiding the costly step of annotating data,
but the quality of the output strongly depends on the type of models
to be recovered.

We consider the latter approach and show how methods from effective
algebraic geometry help finding hidden structure in data that can be
modelled by mixtures of Gaussian distributions.
The algebraic-geometric tool that we consider is tensor decomposition. It
consists in decomposing a tensor into a minimal sum of rank-1
tensors. This decomposition generalises the rank decomposition of a
matrix, with specific and interesting features. Contrarily to matrix
rank decomposition, the decomposition of a tensor is usually unique
(up to permutations) when the rank of the tensor, that is the minimal
number of rank-1 terms in a decomposition, is small compared to the
dimension of the space(s) associated to the tensor.
Such a tensor is called {\em identifiable}.
This property is of particular importance when the decomposition is
used to recover the parameters of a model. It guaranties the
validity of the recovering process and its convergence when the number of data increases.

It has been shown in \cite{Chiantini_2016} that for symmetric tensors,
if the rank of the tensor is strictly less than the rank $r_{g}$ of a
generic tensor of the same size, then the tensor is generically
identifiable, except in three cases.  We show in
\Cref{thm:identifiable} a more specific result: for a symmetric tensor
$T$ having {\em a} decomposition with $r$ points, if the
Hankel matrix associated to $T$ in a degree strictly bigger than the
degree of interpolation of the $r$ points is of rank $r$, then the tensor is
identifiable. We show in \Cref{prop:M3}, that under some assumption on
the spherical gaussian mixtures, a tensor of moments of order 3 of the
distribution is identifiable and its decomposition allows to recover
the parameters of the Gaussian mixture.

Several types of method have been developed to tackle the difficult
problem of tensor decomposition.
Direct methods based on simultaneous
diagonalisation of matrices built from slices of tensors have been
investigated for 3rd order multilinear tensors, e.g. in
\cite{harshman70,
  SanchezTensorialresolutiondirect1990,
  LeurgansDecompositionThreeWayArrays1993a,
  DomanovCanonicalPolyadicDecomposition2014}
or for multilinear tensors of rank smaller than the lowest dimension  in
\cite{DeLathauwerLinkCanonicalDecomposition2006, LucianiCanonicalPolyadicDecomposition2014}.
In his proof on lower bounds of tensor ranks, Strassen showed in
\cite[Theorem 4.1]{StrassenRankoptimalcomputation1983} that a
3rd order multilinear tensor is of rank $r$ if it can be embedded into a
tensor with slices of rank $r$ matrices, which are simultaneously diagonalised.

For symmetric tensor decomposition, a method based on flat extension of Hankel
matrices or commutation of multiplication operators has been proposed
in \cite{BrachatSymmetrictensordecomposition2010a} and extended to
multi-symmetric tensors in
\cite{BernardiGeneraltensordecomposition2013}.
This approach is closely related to the simultaneous diagonalisation of
tensor slices, but follows a more algebraic perspective.
Eigenvectors of symmetric tensors have been used to compute their decompositions in
\cite{OedingEigenvectorstensorsalgorithms2013a}. In
\cite{harmouch:hal-01440063}, Singular Value Decomposition and
eigenvector computation are used to decompose a symmetric tensor, when
its rank is smaller than the smallest size of its Hankel matrix in degree less than half
the order of the tensor.
In \Cref{sec:3}, we describe a new algorithm, involving Singular Value
Decomposition and simultaneous diagonalisation, to compute the decomposition of an
identifiable tensor, which interpolation degree is smaller that half
the order of the tensor.

Numerical methods such as homotopy continuation have been applied
to tensor decomposition in
\cite{HauensteinHomotopytechniquestensor2019, BernardiTensordecompositionhomotopy2017}.
Distance minimisation methods to compute low rank
approximations of tensors have also been investigated.
Alternating Least Squares (ALS) methods, updating alternately the
different factors of the tensor decomposition, is a popular approach
(see e.g.
\cite{Carroll1970, 3923712d8c7b4f8da616f88a35b3f34b,harshman70,KoBa09}), but suffers
from a slow convergence \cite{EMH,doi:10.1137/110843587}.
Other iterative methods such as quasi-Newton methods have been
considered to improve the convergence speed. See
e.g. \cite{Hayashi1982,
  doi:10.1080/10618600.1999.10474853,doi:10.1137/100808034,doi:10.1137/090763172,doi:10.1137/120868323,10.1016/j.csda.2004.11.013,Breiding2018}
for multilinear tensors.
A Riemannian Newton iteration for symmetric tensors is presented in \cite{KHOUJA2022175}.
In \cite{KileelSubspacepowermethod2019}, a method for decomposing real
even-order symmetric tensors, called Subspace Power Method (SPM), and
similar to the power method for matrix eigenvector computation, is proposed.
In these methods, the choice of the initial decomposition is
crucial. In the applications of these algorithms,
the initial point is often chosen at random, yielding approximate
decompositions which can hardly be controlled.
Tensor decomposition methods have numerous applications
\cite{KoBa09}. Some of them
were exploited more recently in Machine Learning.
In \cite{HsuLearningmixturesspherical2013}, symmetric tensor
decompositions for moment tensors are studied for spherical
Gaussian mixtures.
Moment methods have been further investigated for Latent
Dirichlet Allocation models, topic or multiview models
in
\cite{JMLR:v15:anandkumar14b,
JanzaminSpectralLearningMatrices2019a}.
In \cite{ruffini2017clustering}, a tensor decomposition technique
based on Alternate Least Squares (ALS) is used to initialise the Expectation
Maximisation (EM) algorithm, for a mixture of discrete distributions
(which are not Gaussian distributions).
An overview of tensor decomposition methods in Machine Learning can be
found in \cite{RabanserIntroductionTensorDecompositions2017}.

After reviewing Gaussian mixtures and moment methods in
\Cref{sec:2}, we present in \Cref{sec:3} an algebraic symmetric tensor
decomposition method for identifiable tensors. In \Cref{sec:4}, we
apply this algorithm for recovering Gaussian mixtures and
show its impact on providing good initialisation point in the EM
algorithm, in comparison with other state-of-the-art approaches.

\section{Gaussian mixtures and high order moments}\label{sec:2}

In this section, we review Gaussian mixture models and their applications to clustering.

\subsection{Gaussian mixtures}

Suppose that we wish to deal with some Euclidean data $x \in \mathbb{R}^m$, coming from a population composed of $r$ homogeneous sub-populations (often called \emph{clusters}). A reasonable assumption is then that each sub-population can be modelled using a simple probability distribution (e.g. Gaussian). This idea is at the heart of the notion of \emph{mixture distribution}. The prime example of mixture is the \emph{Gaussian mixture}, whose probability density over $\mathbb{R}^m$ is defined as
\begin{equation}
\label{eq:mixture}
    p_\theta(x) = \sum_{j=1}^r \omega_j \mathcal{N}(x | \mu_j, \Sigma_j),
\end{equation}
where $\mathcal{N}( \cdot | \mu, \Sigma)$ denotes the Gaussian density
with mean $\mu \in \mathbb{R}^m$ and definite positive covariance matrices $\Sigma \in \mathcal{S}_m^{++}$. The mixture is parametrised by a typically unknown $\theta = (\omega_1,...,\omega_r,\mu_1,...,\mu_r,$\\$\Sigma_1,...,\Sigma_r)$, composed of
\begin{itemize}
    \item $\omega = (\omega_1,...,\omega_r)$, that belong to the $r$-simplex and correspond to the cluster proportions,
    \item $\mu_j$ and $\Sigma_j$, that correspond respectively to the mean and covariance of each cluster $j \in \{1,...,r\}$.
\end{itemize}
Gaussian mixtures are ubiquitous objects in statistics and machine learning, and own their popularity to many reasons. Let us briefly mention a few of these.

\paragraph{Density estimation} If $r$ is allowed to be sufficiently large, it is possible to approximate any probability density using a Gaussian mixture (see e.g. \cite{nguyen2020approximation}). This motivates the use of Gaussian mixtures as powerful density estimators that can be subsequently used for downstream tasks such as missing data imputation \cite{di2007imputation}, supervised classification \cite{hastie1996discriminant}, or image classification \cite{sanchez2013image} and denoising \cite{houdard2018high}.

\paragraph{Clustering} Perhaps the most common use of Gaussian mixtures is \emph{clustering}, also called \emph{unsupervised classification}. The task of clustering consists in uncovering homogeneous groups among the data at hand. Within the context of Gaussian mixtures, each group generally corresponds to a single Gaussian distribution, as in Equation \eqref{eq:mixture}. If the parameters of a mixture are known, then each point may be clustered using the posterior probabilities obtained via Bayes's rule:
\begin{equation}
\label{eq:bayes}
  \forall x \in \mathbb{R}^m, k \in \{1,...,r\}, \;  \textup{Pr}(x \textup{ belongs to cluster } j ) = \frac{\omega_j \mathcal{N}(x | \mu_j, \Sigma_j)}{p_\theta(x)}.
\end{equation}

Detailed reviews on mixture models and their applications, notably to clustering, can be found in \cite{fraley2002model,bouveyron2019model,mclachlan2019finite}.

\subsection{Learning mixture models}

The main statistical question pertaining mixture models is to estimate the parameters $\theta = (\omega_1,...,\omega_r,\mu_1,...,\mu_r,\Sigma_1,...,\Sigma_r)$ based on a data set $x_1,...,x_n$. Typically, $X_1,...,X_n$ are assumed to be independent and identically distributed random variables with common density $p_\textup{data}$. The problem of statistical estimation is then to find some $\theta$ such that $p_\theta \approx p_\textup{data}$. There are many approaches to this question, the most famous one being the \emph{maximum likelihood method}. Maximum likelihood is based on the idea that maximising the \emph{log-likelihood function}
\begin{equation}
    \ell(\theta) = \sum_{i=1}^n \log p_\theta(x_i),
\end{equation}
will lead to appropriate values of $\theta$. One heuristic reason of the good behaviour of maximum likelihood is that $\ell(\theta)$ can be seen as a measure of how likely the observed data is, according to the mixture model $p_\theta$. This means that the maximum likelihood estimate will be the value of $\theta$ that renders the observed data the likeliest. Another interesting interpretation of maximum likelihood in information-theoretic: when $n \longrightarrow \infty$, maximising the log-likelihood is equivalent to minimising the Kullback-Leibler divergence (an information-theoretic measure of distance between probability distributions) between $p_\theta$ and $p_\textup{data}$, thus giving a precise sense to the statement $p_\theta \approx p_\textup{data}$ (see e.g. \cite[Section 1.6.1]{bishop2006pattern}). For more details on the properties of maximum likelihood, see e.g. \cite[Section 5.5]{van2000asymptotic}.

In the specific case of a mixture model, performing maximum-likelihood is however complex for several reasons. Firstly, as shown for instance by \cite{le1990maximum}, finding a global maximum is actually often ill-posed in the sense that some problematic values of $\theta$ will
lead to $\ell(\theta) = \infty$ while being very poor models of the
data. While focusing on local rather global maxima will fix this first
issue in a sense, iterative optimisation algorithms are likely to
pursue these unfortunate global maxima. Because of the peculiarities
of mixture likelihoods, the most popular algorithm for maximising
$\ell(\theta)$ is the \emph{expectation maximisation} (EM,
\cite{EMalgo}) algorithm, an iterative algorithm specialised for dealing with log-likelihoods of latent variable models. The EM algorithm is usually preferred to more generic gradient-based optimisation algorithms \cite{xu1996convergence}. In a nutshell, at each iteration, the EM algorithm clusters the data using Equation \eqref{eq:bayes}, and then computes the mean and covariance of each cluster. This iterative scheme is related to another popular clustering algorithm known as $k$-means (the close relationship between the two algorithms is detailed in \cite[Section 9]{bishop2006pattern}). A key issue when using the EM algorithm for a Gaussian mixture is the choice of initialisation. Indeed, a poor choice may lead to degenerate solutions, extremely slow convergence, or poor local optima (see \cite{baudry2015mixtures} and references therein). We will see in this paper that good initial points can be obtained by using another estimation method called the \emph{method of moments} (as was previously noted by \cite{ruffini2017clustering} in a context of mixtures of multivariate Bernoulli distributions).

The \emph{method of moments} is a general alternative to maximum likelihood. The idea is to choose several functions $g_1 : \mathbb{R}^m \longrightarrow \mathbb{R}^{q_1},...,g_d : \mathbb{R}^m \longrightarrow \mathbb{R}^{q_d}$ called \emph{moments}, and to find $\theta$ by attempting to solve the system of equations
\begin{equation}
\label{eq:mom}
    \left\{\begin{matrix}
\mathbb{E}_{x \sim p_\textup{data}}[g_1(x)] =  \mathbb{E}_{x \sim p_\theta}[g_1(x)] \\
... \\
\mathbb{E}_{x \sim p_\textup{data}}[g_d(x)] =  \mathbb{E}_{x \sim p_\theta}[g_d(x)].
\end{matrix}\right.
\end{equation}
Of course, since $p_\textup{data}$ is unknown, solving \eqref{eq:mom} is not feasible. However, one may replace the expected moments by empirical versions, and solve instead
\begin{equation}
\label{eq:mom2}
    \left\{\begin{matrix}
\frac{1}{n} \sum_{i=1}^n g_1(x_i)=  \mathbb{E}_{x \sim p_\theta}[g_1(x)] \\
... \\
\frac{1}{n} \sum_{i=1}^n g_d(x_i) =  \mathbb{E}_{x \sim p_\theta}[g_d(x)].
\end{matrix}\right.
\end{equation}
A very simple example of this, in the univariate $m=1$ case, when $g_1(x) = x$, and $g_2(x) = x^2$. Then, solving \eqref{eq:mom} will ensure that the distributions of the model $p_\theta$ and the data $p_\textup{data}$ have the same mean and variance. However, many very different distributions have identical mean and variance! A natural refinement of the previous idea is to consider also higher-order moments $g_3(x) = x^3, g_4(x) = x^4,...$. This will considerably improve the estimates found using the method of moments. This approach was pioneered by \cite{pearson1894contributions} for learning univariate Gaussian mixtures. In the more general multivariate case $m>1$, following \cite{HsuLearningmixturesspherical2013}, the moments chosen can be tensor products, as we detail in the next section in case of a Gaussian mixture with spherical covariances.

\section{Learning structure from tensor decomposition} \label{sec:3}
In this section, we describe the moment tensors revealing the
structure of spherical Gaussian mixtures and how it can be decomposed
using standard linear algebra operations.

Let $\vb X= (X_{1}, \ldots,X_{m})$ be a set of variables. The ring of
polynomials in $\vb X$ with coefficients in $\CC$ is denoted $\CC[\vb
X]$. The space of homogeneous polynomials of degree $d\in \NN$ is
denoted $\CC[\vb X]_{d}$. We recall that a symmetric
tensor $T$ of order $d$ (with real coefficients) can be represented by
an homogeneous polynomial of degree $d$ in the variables $\vb X$  of the form
$$
T(\vb X)= \sum_{|\alpha|=d} T_{\alpha} {d \choose \alpha} \vb X^{\alpha}
$$
where $\alpha=(\alpha_{1}, \ldots, \alpha_{n})\in \NN^{m}$,
$|\alpha|= \alpha_{1}+ \cdots + \alpha_{m}=d$,
$T_{\alpha}\in \RR$,
${d \choose \alpha}={d!\over \alpha_{1}! \cdots \alpha_{m}!}$,
$\vb X^{\alpha}= X_{1}^{\alpha_{1}}\cdots X_{m}^{\alpha_{m}}$.

A decomposition of $T$ as a sum of $d^{\mathrm th}$ power of linear forms is of the form
\begin{equation}\label{eq:waring:dec}
T(\vb X)= \sum_{i=1}^{r} \omega_{i} (\xi_{i}\cdot \vb X)^{d}
\end{equation}
where $\xi_{i}=(\xi_{i,1}, \ldots, \xi_{i,m})\in \CC^{m}$ and
$(\xi_{i}\cdot \vb X)=\sum_{j=1}^{m} \xi_{i,j} X_{j}$.
When $r$ is the minimal number of terms in such a decomposition, it is
called the rank of $T$ and the decomposition is  called a rank
decomposition (or a Waring decomposition) of $T(\vb X)$.

We say that the decomposition is unique
if the lines spanned by $\xi_{1},\ldots, \xi_{r}$ form a unique set of
lines with no repetition.
In this case, the decomposition of $T$ is unique after normalisation
of the vectors $\xi_{i}$ up to permutation (and sign change when $d$ is even).
A tensor $T$ with a unique decomposition  is called an {\em identifiable} tensor. Then the
Waring decompositions of $T$ are of the form $T(\vb X)=\sum_{i=1}^{r}
\omega_{i} \lambda_{i}^{-d} (\lambda_{i}\, \xi_{i}\cdot \vb X)^{d}$
for $\lambda_{i}\neq 0$, $i \in [r]$.

Given a random variable $x\in \RR^{m}$, its moments are $T_{\alpha}=
\EE[x_{1}^{\alpha_{1}}\cdots x_{m}^{\alpha_{m}}]$ for
$\alpha=(\alpha_{1}, \ldots, \alpha_{m})\in \NN^{m}$. The symmetric
tensor of all moments of order $d$ of $x$ is
$$
\EE[(x \cdot \vb X)^{d}] = \sum_{|\alpha|=d} \EE[x_{1}^{\alpha_{1}}\cdots x_{m}^{\alpha_{m}}] {d \choose \alpha} \vb X^{\alpha}.
$$

\subsection{The structure of the moment tensor}
We aim at recovering the hidden structure a random variable, from the
decomposition of its $d^{\mathrm{th}}$ order moment tensor. This is
possible in some circumstances, that we detail hereafter.

\begin{assumption}\label{ass:gm} The random variable $x\in \RR^{m}$ is
  a mixture of spherical Gaussians of probability density \eqref{eq:mixture}
  with parameters
$\theta = (\omega_1,...,\omega_r,\mu_1,...,\mu_r,\sigma_{1}^{2}
I_{m},,...,\sigma_r^{2} I_{m})$
  such that $r\le m$.
\end{assumption}

\begin{theorem}[\cite{HsuLearningmixturesspherical2013}]\label{thm:Mi} Under the previous assumption, let
\begin{itemize}
\item $\tilde{\sigma}^2$ be the smallest eigenvalue of $\EE[(x -\EE[x])\otimes (x -\EE[x])]$
and $v$ a corresponding unit eigenvector,
\item $M_{1}(\vb X) = \EE[(x \cdot \vb X) (v \cdot (x- \EE[x]))^{2}]$,
\item $M_{2}(\vb X) = \EE[(x \cdot \vb X)^{2}] -\tilde{\sigma}^{2}
  \|\vb X\|^{2}$,
\item $M_{3}(\vb X) = \EE[(x \cdot \vb X)^{3}] -
 3\, \|\vb X\|^{2} M_{1}(\vb X)$.
\end{itemize}
Then $\tilde{\sigma}^{2}= \sum_{i=1}^{r} \omega_{i}\, \sigma_{i}^{2}$
and
\begin{equation}\label{eq:waring:M3}
M_{1}(\vb X)= \sum_{i=1}^{r} \omega_{i}\, \sigma_{i}^{2}\,  (\mu_{i} \cdot \vb X),\quad
M_{2}(\vb X)= \sum_{i=1}^{r} \omega_{i}\,  (\mu_{i} \cdot \vb X)^{2},\quad
M_{3}(\vb X)= \sum_{i=1}^{r}\omega_{i}\, (\mu_{i} \cdot \vb X)^{3}.
\end{equation}
\end{theorem}

To analyse the properties of the decomposition  \eqref{eq:waring:M3},
we introduce the apolar product on tensors:
For two homogeneous polynomials $p(\vb X)=\sum_{|\alpha|=d}{\binom{d}{\alpha}p_{\alpha}\vb X^{\alpha}}$
and $q(\vb X)=\sum_{|\alpha|=d}{\binom{d}{\alpha}q_{\alpha}\vb X^{\alpha}}$
of degree $d$, in $\CC[\vb X]_d$, their apolar product is
$$
\apolar{p,q}_d:=\sum_{|\alpha|=d}{\binom{d}{\alpha}\bar{p}_{\alpha}q_{\alpha}}.
$$
The apolar norm of $p$ is $||p||_d=\sqrt{\apolar{p,p}_d}=\sqrt{\sum_{|\alpha|=d}{\binom{d}{\alpha}\bar{p}_{\alpha}p_{\alpha}}}$.
The apolar product is invariant by a linear change of variables of the
unitary group $U_{m}$:
$\forall u\in U_{m}, \apolar{p(u\, \vb X),q(u\, \vb  X)}_{d}=\apolar{p(\vb X),q(\vb X)}_{d}$.

It also satisfies the following properties.
For $v\in \CC^{m}$, $v(\vb X)^{d}= (v \cdot \vb X)^{d}=(v_1X_1+\cdots+v_mX_m)^d,
p \in
\CC[\vb X]_{d}, q\in \CC[\vb X]_{d-1}$, we have :
\begin{itemize}
    \item $\apolar{(v \cdot \vb X)^{d},p}_{d} = p(\bar{v})$,
    \item $\apolar{p,X_i q}_{d} =\frac{1}{d}\apolar{\partial_{X_i}p, q}_{d-1}$.
\end{itemize}

For an homogeneous polynomial $T$ of degree $d\in \NN$ (or
equivalently a symmetric tensor of order $d$), we define the {\em
  Hankel} operator of $T$ in degree $k\le d$ as the map
$$
H_{T}^{k,d-k}: p \in \CC[\vb X]_{d-k}\mapsto
[\apolar{T, \vb X^{\alpha} \,p}_{d}]_{|\alpha|=k} \in \CC^{s_{k}}
$$
where $s_{k}= {m+k-1 \choose k} = \dim \CC[\vb X]_{k}$ is the number of
monomials of degree $k$ in $\vb X$.
The matrix of $H_{T}^{k,d-k}$ in the basis $(\vb X^{\beta})_{|\beta|=d-k}$ is
$$
H_{T}^{k,d-k}=(\apolar{T,\vb X^{\alpha+\beta}}_d)_{|\alpha|=k, |\beta|=d-k}.
$$
From the properties of the apolar product, we see that
$H_{T}^{1,d-1}:p \mapsto {1 \over d}[\apolar{\partial_{X_{i}}T,
  p}_{d-1}]_{1 \le i\le m}$.
For $\xi\in \CC^{m}$ and $k\in \NN$, let
$\xi^{(k)}=(\xi^{\alpha})_{|\alpha|=k}$.
We also check that if $T=(\xi\cdot \vb X)^{d}$
with $\xi\in \CC^{m}$, then
$H_{(\xi\cdot \vb X)^{d}}^{k,d-k}= \bar{\xi}^{\,(k)}\otimes \bar{\xi}^{\,(d-k)}$
is of rank $1$ and its
image is spanned by the vector $\bar{\xi}^{\,(k)}$.

\begin{proposition}\label{prop:M3} Assume that $r\le m$, $w_i>0$ for $i \in [r]$ and
  $\mu_{1},\ldots,\mu_{r}\in \RR^{m}$ are linearly independent.
  The symmetric tensor $M_{3}(\vb X)$ is
  identifiable, of
  rank $r$ and has a unique Waring decomposition satisfying \eqref{eq:waring:M3}.
\end{proposition}
\begin{proof}
Assume that $M_{3}(\vb X)$ has a decomposition of the form \eqref{eq:waring:M3}.
Since the vector $\mu_{1}, \ldots, \mu_{r}$ are linearly independent,
by a linear change of coordinates in $\Gl_{m}$, we can further assume that
$\mu_{1}= e_{1}, \ldots, \mu_{r}= e_{r}$ are
the first $r$ vectors of the canonical basis of
$\RR^{m}$. In this coordinate system, $M_{3}(\vb X)= \sum_{i=1}^{r}
 X_{i}^{3}$ and the matrix $H_{M_{3}}^{1,2}$ in a
convenient basis has a $r\times r$ identity block and zero elsewhere. Thus $H_{M_{3}}^{1,2}$ is of rank $r$.
Its kernel of dimension ${1\over 2}\, m \, (m+1) -r$ is spanned by
the polynomials $X_{i} X_{j}$ with
$(i,j) \neq (k,k)$ for $k \in [r]$. The kernel of $H_{M_{3}}^{1,2}$ is thus
the space of homogeneous polynomials of degree $2$, vanishing at
$e_{1}, \ldots, e_{r} \in \RR^{n}$.

If $M_{3}(\vb X)$ can be decomposed as $M_{3}(\vb
X)=\sum_{i=1}^{r'}\omega'_{i}\, (\mu'_{i} \cdot \vb X)^{3}$ with
$\omega'_{i}\in \CC$, $\mu'_{i}\in \CC^{m}$ and $r'<r$,
then  $H_{M_{3}}^{1,2}$, as a sum of $r'<r$ matrices $\omega'_{i} H_{(\mu'_{i}\cdot
\vb X)^{3}}^{1,2}$ of rank $1$, would be of rank smaller than $r'<r$, which is a
contradiction. Thus a minimal decomposition of $M_{3}(\vb X)$ is of length
$r$ and $r$ is the rank of $M_{3}(\vb X)$.

Let us show that the decomposition \eqref{eq:waring:M3} of $M_{3}(\vb X)$ is unique up
to a scaling of the vector $\mu_{i}$, i.e. that $M_{3}(\vb X)$ is
identifiable.
For any Waring decomposition $M_{3}(\vb X)=\sum_{i=1}^{r}\omega'_{i}\,
(\mu'_{i} \cdot \vb X)^{3}$, the vectors  $\mu'_{1}, \ldots, \mu'_{r}$ are linear independant, since
$\mu'_{i}$ spans $\textup{im} H_{(\mu'_{i}\cdot \vb X)^{3}}^{1,2}$ and
$H_{M_{3}}^{1,2} = \sum_{i=1}^{r} \omega'_{i}H_{(\mu'_{i}\cdot \vb
  X)^{3}}^{1,2}$ is of rank $r$.
As $\mu'_{1}, \ldots,
\mu'_{r}$ can be transformed into $e_{1}, \ldots, e_{r}$ by a linear change of variables, $\ker
H_{M_{3}}^{1,2}$ is also the vector space of homogeneous polynomials of degree $2$, vanishing at
$\mu'_{1}, \ldots, \mu'_{r} \in \CC^{m}$.
Therefore, the set of $\{\mu'_{1}, \ldots, \mu'_{r}\}$ coincides, up to a scaling, with the
set of points $\{\mu_{1}, \ldots, \mu_{r}\}$ of another Waring
decomposition of $M_{3}(\vb X)=\sum_{i=1}^{r}\omega_{i}\, (\mu_{i} \cdot \vb
X)^{3}$. This shows that $M_{3}(\vb X)$ is identifiable.

Therefore, a Waring decomposition of $M_{3}(\vb X)$ is of the form
$M_{3}(\vb X) = \sum_{i=1}^{r}\tilde{\omega}_{i} \, (\tilde{\mu}_{i} \cdot \vb X)^{3}$
with $\tilde{\omega}_{i} =  \lambda^{-3} \omega_{i}$,
$\tilde{\mu}_{i}= \lambda_{i} \mu_{i}$ and $\lambda_{i} \neq 0$ for $i
\in [r]$.
As $\tilde{\mu}_{1}, \ldots, \tilde{\mu}_{r}$ are linearly independent, the
homogeneous polynomials
$(\tilde{\mu}_{1}\cdot \vb X)^{2}, \ldots, (\tilde{\mu}_{r}\cdot \vb
X)^{2}$ are also linearly independant in $\CC[\vb X]_{2}$ (by a linear
change of variables, they are equivalent to $X_{1}^{2}, \ldots, X_{r}^{2}$).
Consequently, the relation
$$
M_{2}(\vb X)= \sum_{i=1}^{r} {\omega_{i}} ({\mu}_{i}\cdot \vb X)^{2}= \sum_{i=1}^{r} \lambda _{i}\tilde{\omega_{i}} (\tilde{\mu}_{i}\cdot \vb X)^{2}
$$
defines uniquely $\lambda_{1}, \ldots, \lambda_{r}$, and
$M_{3}(\vb X)$ has a unique Waring decomposition, which satisfies the relations \eqref{eq:waring:M3}.
\end{proof}

Under Assumption \ref{ass:gm}, the hidden structure of the random
variable $x$ can thus be recovered using Algorithm \ref{algo:recover}.

\begin{algorithm}[ht]\caption{\label{algo:recover}Recovering the
    hidden structure of a Gaussian mixture}

\strong{Input:} The moment tensors $M_{1}(\vb X), M_{2}(\vb X), M_{3}(\vb X)$.
\begin{itemize}
 \item Compute a Waring decomposition of $M_{3}(\vb X)$ to get
   $\tilde{\omega}_{i} \in \RR, \tilde{\mu}_{i}
   \in \RR^{m}$, $i\in [r]$ such that $M_{3}(\vb X)=\sum_{i=1}^{r}
   \tilde{\omega}_{i}\, (\tilde{\mu}_{i} \cdot \vb X)^{3}$.
 \item Solve the system $\sum_{i=1}^{r}
   \tilde{\omega}_{i}\, (\tilde{\mu}_{i}\cdot \vb X)^{2} \lambda_{i}= M_{2}(\vb X)$ to get
   $\lambda_{i} \in \RR$ and $\omega_{i}= \lambda_{i}^{3} \tilde{\omega}_{i}\in \RR_{+}$,
   $\mu_{i} = \lambda_{i}^{-1}\tilde{\mu_{i}}\in \RR^{m}$ such that
   $M_{3}(\vb X)=\sum_{i=1}^{r} {\omega}_{i}\,  ({\mu}_{i} \cdot \vb X)^{3}$ and
   $M_{2}(\vb X)=\sum_{i=1}^{r} {\omega}_{i} \, ({\mu}_{i} \cdot \vb X)^{2}$.
 \item Solve the system $\sum_{i=1}^{r} \omega_{i}
  (\mu_{i} \cdot \vb X)  \sigma_{i}^{2} = M_{1}(\vb X)$
   to get $\sigma_{i}^2\in \RR_{+}$.
 \end{itemize}
\strong{Output:} $\omega_{i}\in \RR_{+}, \mu_{i} \in \RR^{n}$,
$\sigma_{i}^2\in \RR_{+}$ for $i\in [r]$.
\end{algorithm}
This yields the parameters $\omega_{i}\in \RR_{+}, \mu_{i} \in
\RR^{m}$,
$\sigma_{i}\in \RR_{+}$ for $i\in [r]$ of the Gaussian mixture $x$.

In the experimentation, the moments involved in the tensors $M_{i}$ will be
approximated by empirical moments and we will compute an approximate
decomposition of the empirical moment tensor $\hat{M}_{3}(\vb X)$.

\subsection{Decomposition of identifiable tensors}
We describe now an important step of the approach, which is computing a
Waring decomposition of a tensor.
In this section, we consider a tensor $T\in \CC[\vb X]_{d}$ of order $d\in \NN$ with a
Waring decomposition of the form $T = \sum_{i=1}^{r} \omega_{i}\,
(\xi_{i}\cdot \vb X)^{d}$  with  $\omega_{i} \in \CC, \xi_{i}\in
\CC^{m}$, that we recover by linear algebra techniques, under some hypotheses.

\begin{definition} The interpolation degree $\iota({\Xi})$ of $\Xi=\{\xi_{1},
\ldots, \xi_{r}\} \subset \CC^{m}$ is the smallest degree $k$ of a family of homogenous interpolation
polynomials $u_{1}, \ldots, u_{r}\in \CC[\vb X]_{k}$ at the points $\Xi$
($u_{i}(\xi_{j})= \delta_{i,j}$ for $i,j \in [r]$).
\end{definition}

For any $d\ge\iota(\Xi)$, there exists a family $(\tilde{u}_{i})_{i\in [r]}$ of interpolation
polynomials of degree $d$, obtained from an interpolation family
$(u_{i})_{i\in [r]}$
in degree $\iota(\Xi)$ as $ \tilde{u}_{i} = {(\lambda\cdot \vb
X)^{d-\iota(\Xi)} \over (\lambda\cdot \xi_{i})^{d-\iota(\Xi)}  }u_{i}$
for a generic $\lambda\in \CC^{m}$ such that $\lambda \cdot \xi_{i}
\neq 0$ for $i\in [r]$.

Notice that if the points $\Xi=\{\xi_{1}, \ldots, \xi_{r}\}$ are
linearly independent (and therefore
$r\le m$), then $\iota(\Xi)=1$ since a family of linear forms
interpolating $\Xi$ can be constructed.

If $k \ge \iota(\Xi)$, then the evaluation map
$\vb e_{\Xi}^{(k)}:p\in \CC[\vb X]_{k} \mapsto (p(\xi_{1}), \ldots,
p(\xi_{r}))\in \CC^{r}$ is surjective. Its kernel is the space of homogeneous
polynomials of degree $k$ vanishing at $\Xi$. Any supplementary space
admits a basis  $u_{1}, \ldots, u_{r}$, which is an interpolating
family for $\Xi$ in degree $k$.
A property of the interpolation degree is the following:
\begin{lemma}\label{lem:kereval}
For $k> \iota(\Xi)$, the common roots of $\ker \vb e^{(k)}_{\Xi}$ is
the union $\cup_{i=1}^{r} \CC\, \xi_{i}$ of lines spanned by $\xi_{1}, \ldots, \xi_{r}\in \CC^{m}$.
\end{lemma}
\begin{proof}
As $\iota(\Xi)+1$ is the
Castelnuovo-Mumford regularity of the vanishing ideal $I(\Xi)=\{p \in
\CC[\vb X] \mid p \textup{ homogeneous, } p(\xi) =0 \textup{ for } \xi \in \Xi\}$ \cite{eisenbud_geometry_2005}[Ch.4],
it is generated in degree $k>\iota(\Xi)$ and the common roots of
$\ker \vb e^{(k)}_{\Xi}=I(\Xi)_{k}$ is $\cup_{i=1}^{r} \CC\, \xi_{i}$.
\end{proof}

Hereafter, we show that tensors $T$ such that $\rank H_{T}^{k,d-k} = r$
for $k>\iota(\Xi)+1$ are identifiable and we describe a numerically robust algorithm to
compute their Waring decomposition.

Let $U=(U_{\alpha, j})_{|\alpha|=k, j\in [r]} \in \CC^{s_{k}\times r}$
be such that $\im U = \im H_{T}^{k,d-k}$ and
$U_{i}= (U_{e_{i}+\alpha, j})_{|\alpha|=k-1, j\in [r]}$ be the
submatrices of $U$ with the rows indexed by the monomials divisible by
$X_{i}$ for $i \in [m]$.

\begin{theorem}\label{thm:identifiable}
Let $T\in \CC[\vb X]_{d}$ with a decomposition $T =
\sum_{i=1}^{r} \omega_{i}\, (\xi_{i}\cdot \vb X)^{d}$  with
$\omega_{i} \in \CC$ and $\xi_{i}=(\xi_{i,1}, \ldots, \xi_{i,n})\in
\CC^{m}$ such that $\rank H_{T}^{k,d-k} = r$ for some $k\in
[\iota(\xi_{1}, \ldots, \xi_{r})+1, d]$. Then $T$ is identifiable of rank $r$
and there exist invertible matrices $E\in \CC^{s_{k}\times s_{k}}$,
$F\in \CC^{r\times r}$ such that
\begin{equation}\label{eq:simdiag}
E^{t}\, U_{i}\, F = \left[ \begin{array}{c}
\Delta_{i}\\
0
\end{array}
\right]
\end{equation}
with $\Delta_{i}=\diag(\bar{\xi}_{1,i}, \ldots, \bar{\xi}_{r,i})$ for $i\in [m]$.
For any pair $(E,F)$, which diagonalises simultaneously $[U_{1},
\ldots, U_{m}]$ as in \eqref{eq:simdiag}, there exist unique $\omega'_{1}, \ldots,
\omega'_{r}\in \CC$ such that $T = \sum_{i=1}^{r} \omega'_{i}\, (\xi'_{i}\cdot \vb X)^{d}$
with $\bar{\xi}'_{i}= ( (\Delta_{1})_{i,i}, \ldots, (\Delta_{m})_{i,i})$.
\end{theorem}
\begin{proof}
From the decomposition of $T$, we have
for
$k\le d$ that
$$
H_{T}^{k,d-k} = \sum_{i=1}^{r} \omega_{i} \, \bar{\xi}_{i}^{\,(k)} \otimes \bar{\xi}_{i}^{\,(d-k)}
$$
is a linear combination of $r$ Hankel matrices $\bar{\xi}_{i}^{\,(k)}
\otimes \bar{\xi}_{i}^{\,(d-k)}$ of rank $1$.
If $T$ is of rank $r'<r$, then using its decomposition of rank
$r'$,  $H_{T}^{k,d-k}$ would be of rank $\le r'<r$, which is a
contradiction. This shows that $T$ is of rank $r$.

As $\rank H_{T}^{k,d-k}=r$,
we deduce that the image of $H_{T}^{k,d-k}$ is spanned by
$\bar{\xi}_{1}^{\,(k)}, \ldots, \bar{\xi}_{r}^{(k)}$ and there exists
an invertible matrix $F\in \CC^{r \times r}$ such that
$$
U \, F  = [\bar{\xi}_{1}^{\,(k)}, \ldots, \bar{\xi}_{r}^{(k)}]
$$
For any polynomial $p\in \CC[\vb X]_{k}$, which coefficient vector in
the monomial basis $(\vb X^{\alpha})_{|\alpha|=k}$ is denoted $[p]$,
we have
$[p]^{t} U F = [p(\bar{\xi}_{1}), \ldots, p(\bar{\xi}_{r})]^{t}$.
This shows that $U^{\perp}=\{p\in \CC[\vb X]\mid
[p]^{t} U = 0\}$ is $\ker \vb e_{\bar{\Xi}}^{(k)}$.
By Lemma \ref{lem:kereval} since $k\ge  \iota(\bar{\Xi})$, the common roots of
the homogeneous polynomials in $\ker \vb e_{\bar{\Xi}}^{(k)}$ are the
scalar multiples of $\bar{\Xi}$.
Consequently, the set of lines spanned by the vectors ${\Xi}$ of a
Waring decomposition of $T$ is uniquely determined as the conjugate of the
zero locus of $U^{\perp}\subset \CC[\vb X]_{k}$ and $T$ is identifiable.

For any $p\in \CC[\vb X]_{k-1}$ represented by its coefficient
vector $[p]$ in the monomial basis $(\vb X^{\alpha})_{|\alpha|=k-1}$, we have
\begin{equation}\label{eq:eval}
[p]^{t} U_{i} F = [x_{i} p]^{t} U F = [\bar{\xi}_{1,i}\,
p(\bar{\xi}_{1}), \ldots, \bar{\xi}_{r,i}\, p(\bar{\xi}_{r})]^{t}.
\end{equation}

Let $E$ be the coefficient matrix of a basis $u_{1}, \ldots, u_{r},
v_{r+1}, \ldots, v_{s_{k-1}}$ of $\CC[\vb X]_{k-1}$, such that $u_{1},
\ldots, u_{r}$ is an interpolating family for
$\bar{\Xi}=\{\bar{\xi}_{1},\ldots,\bar{\xi}_{r}\}$
and $v_{r+1},\ldots, v_{s_{k-1}}$ is a basis of $\ker \vb
e_{\bar{\Xi}}^{(k-1)}$. The matrix $E$ is invertible by construction, and we
deduce from \eqref{eq:eval} that
$$
E^{t} U_{i} F = \left[
\begin{array}{c}
  \diag(\bar{\xi}_{1,i}, \ldots, \bar{\xi}_{r,i})\\
  0
\end{array}
\right].
$$

Let us show conversely that for any pair of matrices $(E',F')$, which
diagonalises simultaneously $[U_{1}, \ldots, U_{m}]$ as in \eqref{eq:simdiag} with $\Delta_{i}= \diag(\bar{\xi}'_{1,i}, \ldots, \bar{\xi}'_{r,i})$, there exist unique
$\omega'_{1}, \ldots, \omega'_{r}\in \CC$ such that $T = \sum_{i=1}^{r}
\omega'_{i}\, (\xi'_{i}\cdot \vb X)^{d}$.

Let $u'_{1}, \ldots, u'_{r}, v'_{r+1}, \ldots, v'_{s_{k-1}} \in
\CC[\vb X]$ be the polynomials corresponding to the columns of
$E'$. Then for a generic $\lambda = (\lambda_{1}, \ldots, \lambda_{r})\in
\CC^{m}$, we have
\begin{eqnarray*}
\diag((\lambda \cdot \bar{\xi}'_{1}), \ldots, (\lambda \cdot
  \bar{\xi}'_{r})) &=&
\sum_{i=1}^{m} \lambda_{i} [u'_{1},\ldots, u'_{r}]^{t} U_{i} F' =
\sum_{i=1}^{m} \lambda_{i} [u'_{1}, \ldots, u'_{r}]^{t} U_{i} F (F^{-1} F') \\
  &=&
[(\lambda \cdot \bar{\xi}_{j})\, u'_{i}(\bar{\xi}_{j})]_{i,j\in [r]} F^{-1} F'\\
&=&\diag((\lambda \cdot \bar{\xi}_{1}), \ldots, (\lambda \cdot
    \bar{\xi}_{r}))\,
[u'_{i}(\bar{\xi}_{j})]_{i,j\in [r]} F^{-1} F'.
\end{eqnarray*}
As $\lambda\in \CC^{m}$ is generic and $\lambda \cdot \bar{\xi}_{i}\neq 0$
for $i\in [r]$, we deduce that $\Delta=
[u'_{i}(\bar{\xi}_{j})]_{i,j\in [r]} F^{-1} F'$ is a diagonal and
invertible matrix
and that $\xi'_{i}= \bar{\Delta}_{i,i} \xi_{i}$ with $\Delta_{i,i}\neq
0$.

Then we have $(\xi'_{i}\cdot \vb X)^{d} =
\bar{\Delta}_{i,i}^{d}\, (\xi_{i}\cdot \vb X)^{d}$
and
$T = \sum_{i=1}^{r} \omega'_{i}\, (\xi'_{i}\cdot \vb X)^{d}$
with $\omega'_{i}= \bar{\Delta}_{i,i}^{-d} \omega_{i}$, which concludes the proof of
the theorem.
\end{proof}

This leads to Algorithm \ref{algo:decomp} to compute a Waring
decomposition of an identifiable tensor $T$.
\begin{algorithm}[ht]\caption{\label{algo:decomp}Decomposition of
 an identifiable tensor}
\strong{Input:} $T\in \CC[\vb X]_{d}$, which admits a decomposition with
$r$ points $\Xi=\{\xi_{1}, \ldots, \xi_{r}\}$ and $k > \iota(\Xi)$.

\begin{itemize}
 \item Compute the Singular Value Decomposition of $H_{T}^{k,d-k} = U\,
   S\, V^{t}$;
 \item Deduce the rank $r$ of $H_{T}^{k,d-k}$, take the first $r$
   columns of $U$ and build the submatrices $U_{i}$ with rows indexed by the monomials $(X_{i} \vb
   X^{\alpha})_{|\alpha|=k-1}$ for $i\in [n]$;
 \item Compute a simultaneous diagonalisation of the pencil $[U_{1}
   \ldots,U_{m}]$
   as $E^{t} U_{i} F = \left[
\begin{array}{c}
  \diag(\bar{\xi}_{1,i}, \ldots, \bar{\xi}_{r,i})\\
  0
\end{array}
\right]$ and deduce the points $\xi_{i}=(\xi_{i,1}, \ldots,
\xi_{i,m})\in \CC^{m}$ for $i\in [r]$;
 \item Compute the weights $\omega_{1}, \ldots, \omega_{r}$ by solving
   the linear system $T = \sum_{i=1}^{r} \omega_{i}\, (\xi_{i}\cdot \vb X)^{d}$;
 \end{itemize}

\strong{Output:} $\omega_{i}\in \CC$, $\xi_{i}\in \CC^{m}$ s.t. $T = \sum_{i=1}^{r} \omega_{i}\, (\xi_{i}\cdot \vb X)^{d}$.
\end{algorithm}
\section{Numerical experimentations}\label{sec:4}
The model used in this section is the Gaussian Mixture Model (GMM) with differing spherical covariance matrices. Recall that if $x=(x_1, \dots, x_n)$ is a sample of $n$ independent observations from $r$ multivariate Gaussian mixture with differing spherical covariance matrices of dimension $m$, and $h = (h_1,h_2,\ldots,h_n)$ is the latent variable that determine the component from which the observation originates, then:
\begin{equation*}
x_{i}\mid (h_{i}=k)\sim {\mathcal {N}}_{m}({\mu }_{k},\sigma_k^2 I_m)
~\text{where,}
\end{equation*}
\begin{equation*}
\mathrm{Pr}(h_{i}=k)=\omega_k,~\text{for}~k\in[r], ~\text{such that}~\sum_{k=1}^{r}{\omega_k=1.}
\end{equation*} The aim of statistical inference is to find the unknown parametrs $\mu_k$, $\sigma_k^2$ and $w_k$, for $k\in [r]$ from the data $x$. This can be done by finding the maximum likelihood estimation (MLE) i.e. finding the optimal maximum of the likelihood function associated to this model. The expectation maximisation algorithm (EM) \cite{EMalgo}, usually used for finding MLEs, is an iterative algorithm in which the initialisation i.e. the initial estimation of the latent parameters is crucial, since various initialisations can lead to different local maxima of the likelihood function, consequently, yielding different clustering partition. Thus, in this section we compare the clustering results obtained by different initialisation of the EM algorithm against the initialisation by the method of moments through examples of simulated (subsection \ref{simul}) and real (subsection \ref{real}) datasets. We fix a maximum of 100 iterations of the EM algorithm. The different initialisation considered in this section are the following:
\begin{itemize}
    \item The k-means method \cite{k-means} according to the following strategy:\\
    The best partition obtained out of 50 runs of the k-means algorithm.
    \item The method of moments, where \Cref{algo:recover} is applied to build the moments and \Cref{algo:decomp} is applied
      to the empirical moment tensor corresponding to $M_{3}(\vb X)$
      (see \Cref{thm:Mi}),
      with less than 5 Riemannian Newton iterations \cite{KHOUJA2022175} to
      reduce the distance between the empirical moment tensor and its decomposition.
    \item The Model-based hierarchical agglomerative clustering algorithm (MBHC) \cite{MBHC, MBHC1}.
    \item The emEM strategy \cite{emEM} as in \cite{Emem1} which makes 5 iterations for each of 50 short runs of EM, and follows the one which maximises the log-likelihood function by a long run of EM.
\end{itemize}
The k-means, MBHC and emEM are common strategies for initialising the EM algorithm for GMMs. The comparison among the different EM initialisation strategies is based on three measures: The  Bayesian Information Criterion (BIC) \cite{BIC,BIC1},  the Adjusted Rand Index (ARI) \cite{ARI}, and the error rate (errorRate). The BIC is a penalized-likelihood criterion given by the following formula
\begin{equation*}
    \mathrm{BIC}=-2 \ell(\hat{\theta})+\log (n)\nu,
\end{equation*} where $\ell$ is the log-likelihood function , $\hat{\theta}$ is the MLE which maximises the log-likelihood function and $\nu$ is the number of the estimated parameters. This criterion measures the quality of the model such that for comparing models the one with the largest BIC value among the other models is the most fitted to the studied dataset. The ARI criterion measures the similarity between the estimated clustering obtained by the applied model and the exact true clustering. Its value is bounded between 0 and 1. The more this measure is close to 1 the more the estimated clustering is accurate. The error rate measure can be viewed as an alternative of the ARI. In fact this criterion measures the minimum error between the predicted clustering and the true clustering, and thus low error rate means high agreement between the estimated and the true clustering. The former criteria as well as the EM algorithm are used from the tools of the package \texttt{mclust} \cite{mclust} in \textsf{R} programming language.

\subsection{Simulation}\label{simul}
We performed 100 simulations from each of the two models described in examples \ref{ex1} and \ref{ex2}. We counted the instances where each of the considered initialising strategies for the EM could find throughout the 100 simulated data and among the other initialisation methods the largest BIC, the highest ARI, ARI$\geq0.99$ (as in this case the clustering obtained is the most accurate) and the lowest errorRate. The values of the BIC, ARI, errorRate and consumed time of the different considered initialisation strategies for one dataset sampled according to the model of Example \ref{ex1} (resp. \ref{ex2}) are presented in Table \ref{table1} (resp. \ref{table2}), and Figure \ref{fig1} (resp. \ref{fig2}) shows a two-dimensional visualisation of the observations according to the first four features, the observations in the upper panels are labeled according to the actual clustering, while they are labeled in the lower panels according to the clustering obtained by the EM algorithm initialised by the method of moments. In order to have an estimation about the numerical stability of the obtained results, we repeat the same numerical experiment for each example 20 times and we compute the means (Table \ref{table5}, \ref{table6}) and the variances (values in parentheses in Table \ref{table5}, \ref{table6}) of the 20 percentages obtained of each of the BIC, ARI, ARI$\geq0.99$ and errorRate values for the different initialising strategies.\\
As we mentioned before the initialisation strategies considered in this comparison against the method of moments are common and have, in general, good numerical behavior. Nevertheless, we cannot expect all the initialisation strategies that exist for the EM algorithm to work well in all the cases \cite{emEM, em-initi1}. Hereafter, two examples are chosen in such a way to present some cases where the common initialisation strategies k-means, MBHC and emEM have some difficulties to provide a good initialisation to the EM algorithm for the GMMs with differing spherical covariance matrices, or in other words where the initialisation by the method of moments outperforms the other considered initialisations. For instance, we put in each of these two examples one cluster of small size (the blue cluster in Figure \ref{fig1}, the red cluster in Figure \ref{fig2}), we want to make the clusters overlap, since these initialisation strategies could misscluster the dataset if the clusters are intersecting. We notice that this choice of the mean vectors and the different variances in each of the two examples yields a dataset with the expected clustering characteristic.

\begin{example} \label{ex1}
In the first simulation example, a multivariate dataset (m=6) of n=1000 observations generated with r=4 clusters according to the following parameters:
\begin{itemize}
\item The probability vector: $\omega={(0.2782,0.0139,0.3324,0.3756)}^T$.
\item The mean vectors: $\mu_1={(-5.0, -9.0, 8.0, 8.0, 2.0, 5.0)}^T$, $\mu_2={(-7.0, 6.0, -1.0, 6.0, -8.0, -10.0)}^T$, $\mu_3={(-4.0, -10.0, -5.0, 1.0, 5.0, 4.0)}^T$, $\mu_4={(-6.0, 6.0, 5.0, 4.0, -1.0, -1.0)}^T$.
\item The variances: $\sigma_1^2=1.5$, $\sigma_2^2=2.5$, $\sigma_3^2=5.0$, $\sigma_4^2=15.0$.
\end{itemize}
\begin{table}[ht]
  \centering
  \caption{Numerical results of one data set of Example \ref{ex1}}
    \begin{tabular}{|c|c|c|c|c|}\hline
    Method&BIC&ARI&errorRate&time(s)\\\hline
    em\_km&-29590.48&0.8281&0.168&\textbf{0.045} \\
    em\_mom&\textbf{-29492.11}&\textbf{1.0}&\textbf{0.0}&0.547 \\
    em\_mbhc&-29594.97&0.8574&0.099&0.287 \\
    em\_emEM&-29593.18&0.8366&0.132&0.171 \\\hline
    \end{tabular}%
  \label{table1}%
\end{table}%
\begin{figure}[ht]
\centerline{\includegraphics[width=5in, height=5in]{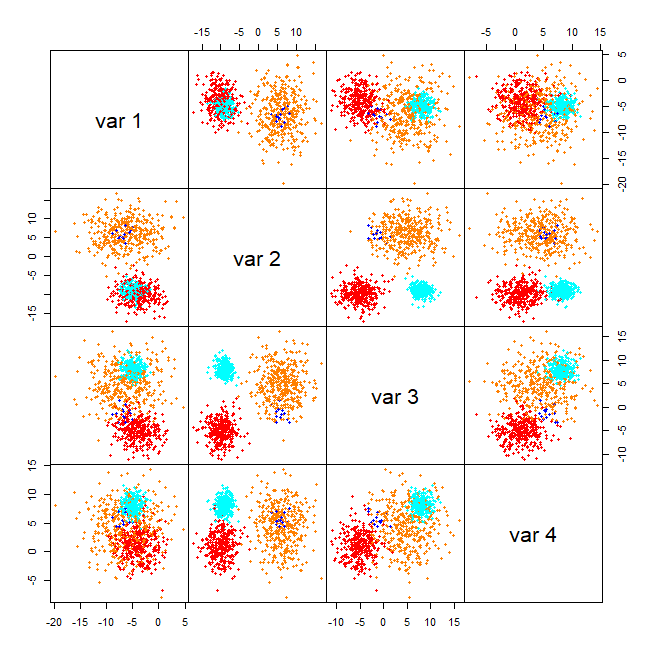}}
\caption{Scatterplot matrix for the sampled dataset of Example \ref{ex1} projected onto the first four variables (features): upper panels show scatterplots for pairs of variables in the
original clustering; lower panels show the clustering obtained by applying the EM algorithm initialised by the method of moments.}
\label{fig1}
\end{figure}
\begin{table}[ht]
  \centering
  \caption{Estimation of the stability of Example \ref{ex1} results }
    \begin{tabular}{|c|c|c|c|c|}\hline
    Method&BIC&ARI&$\text{ARI}\geq 0.99$&errorRate \\\hline
    em\_km&38.35\% (37.82)&47.6\% (21.41)&48.85\% (21.61)&47.6\% (21.2) \\
    em\_mom&\textbf{74.8\%} (41.01)&\textbf{88.75\%} (15.36)&\textbf{83.4\%} (18.36)&\textbf{88.60\%} (14.46) \\
    em\_mbhc&10.75\% (12.41)&15.9\% (17.57)&15.55\% (22.99)&15.9\% (19.46) \\
    em\_emEM&7.3\% (8.43)&14.5\% (8.05)&12.6\% (17.83)&14.95\% (7.52) \\\hline
    \end{tabular}%
  \label{table5}%
\end{table}%

\end{example}
\begin{example} \label{ex2}
In the second simulation example, a multivariate dataset (m=5) of n=1000 observations generated with r=3 clusters according to the following parameters:
\begin{itemize}
\item The probability vector: $\omega={(0.0930, 0.2151, 0.6918)}^T$.
\item The mean vectors: $\mu_1={(7.0, -4.0, -4.0, -6.0, -4.0)}^T$, $\mu_2={(2.0, -4.0, -6.0, -10.0, -3.0)}^T$, $\mu_3={(4.0, -4.0, -5.0, 6.0, 1.0)}^T$.
\item The variances: $\sigma_1^2=5.0$, $\sigma_2^2=10.0$, $\sigma_3^2=15.0$.
\end{itemize}
\begin{table}[ht]
  \centering
  \caption{Numerical results of one data set of Example \ref{ex2}}
    \begin{tabular}{|c|c|c|c|c|}\hline
     Method&BIC&ARI&errorRate&time(s) \\\hline
    em\_km&-28360.30&0.4352&0.309&\textbf{0.051} \\
    em\_mom&\textbf{-28246.02}&\textbf{0.9498}&\textbf{0.03}&0.504 \\
    em\_mbhc&-28358.67&0.3197&0.384&0.292 \\
    em\_emEM&-28360.42&0.4408&0.296&0.141 \\\hline
    \end{tabular}%
  \label{table2}%
\end{table}%

\begin{figure}[ht]
\centerline{\includegraphics[width=5in, height=5in]{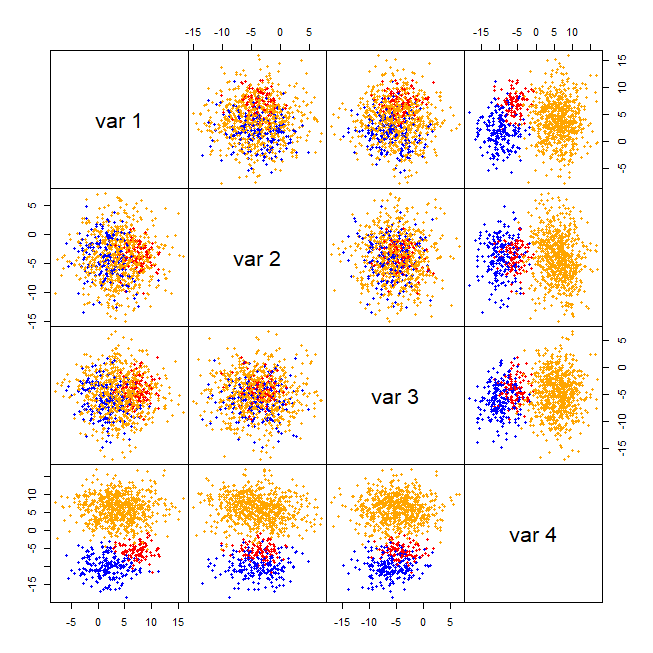}}
\caption{Scatterplot matrix for the sampled dataset of Example \ref{ex2} projected onto the first four variables (features): upper panels show scatterplots for pairs of variables in the
original clustering; lower panels show the clustering obtained by applying the EM algorithm initialised by the method of moments.}
\label{fig2}
\end{figure}
\begin{table}[ht]
  \centering
  \caption{Estimation of the stability of Example \ref{ex2} results}
    \begin{tabular}{|c|c|c|c|c|}\hline
    Method&BIC&ARI&$\text{ARI}\geq 0.99$&errorRate \\\hline
    em\_km&0.45\% (0.576) &0.05\% (0.05) &0.0\% (0.0)&0.1\%(0.095) \\
    em\_mom&\textbf{50.0\%} (18.63)&\textbf{92.35\%} (9.82)&\textbf{0.0\%} (0.0)&\textbf{92.1\%} (7.46) \\
    em\_mbhc&49.35\% (19.82)&2.45\% (3.63)&0.0\% (0.0)&2.45\% (2.58) \\
    em\_emEM&0.3\% (0.326)&5.2\% (4.48)&0.0\% (0.0)&5.9\% (5.36) \\\hline
    \end{tabular}%
  \label{table6}%
\end{table}

\end{example}

The Table \ref{table5}, \ref{table6} show that in Example \ref{ex1}, \ref{ex2} the best results among the considered initialising strategies are for the method of moments. In fact, in the former two tables we see that the method of moments found throughout the 100 simulated datasets, in average (by runing the numerical experiment 20 times), the largest BIC, highest ARI, ARI$\geq 0.99$ and lowest errorRate among the other initialisation strategies in more instances than all the other considered initialisation method, implying in this context marked outperformance for the moments initialisation method. Note that the consumed time (see. Table \ref{table1}, \ref{table2}) tends to be higher in the method of moments than in the other initialisation strategies. This is expected since stochastic approaches (to which the methods k-means, MBHC and emEM belong) outperform the deterministic approaches (as the method of moments) in this term.
\clearpage

\subsection{Real data}\label{real}
In this subsection we present four examples of real datasets, for which we know already their number of clusters, and we report the different BIC, ARI and errorRate values as well as the consumed time attained by the EM algorithm initialised by the different considered initialisation strategies and used with the GMM of different spherical covariance matrices. The explored real data are: The famous iris data \cite{iris, iris1} widely used as an example of clustering to test the algorithms, Diabetes \cite{diabetes}, olive oil \cite{oliveoil}, and MNIST \cite{deng2012mnist}.
\begin{example}[Iris] \label{iris}
The iris dataset contains four physical measurements (length and width of sepals and petals) for 50 samples of three species of iris (setosa, virginica and versicolor). The number of features is $m=4$ and the number of clusters is $r=3$.\\
\begin{table}[ht]
  \centering
  \caption{Numerical results of Example \ref{iris}}
    \begin{tabular}{|c|c|c|c|c|}\hline
    Method&BIC&ARI&errorRate&time(s) \\\hline
    em\_km&-1227.6656 &0.6199 &0.167 &\textbf{0.007} \\
    em\_mom& \textbf{-1227.6676}&\textbf{0.6410} & \textbf{0.153} &0.203 \\
    em\_mbhc&-1227.6696 &0.6199 &0.167 &\textbf{0.007}  \\
    em\_emEM&-1227.6495 &0.6302 &0.160 &0.045 \\\hline
    \end{tabular}%
  \label{table7}
\end{table}%
The four initialisation strategies yield the same BIC value. The ARI and the errorRate values are slightly better with the moment initialisation among the other considered initialisation strategies. On the other hand, the consumed time is clear higher in the moment method initialisation.
\end{example}
\begin{example}[Diabetes]\label{diab}
The Diabete dataset \cite{diabetes} contains three measurements: glucose, insulin and sspg; made on 145 non-obese adult patients classified into three types of diabetes: Normal, Overt, and Chemical. Herein, in this example $m=r=3$. We apply the different initialisation strategies for the EM algorithm, the \Cref{table8} shows the results.
\begin{table}[H]
  \centering
  \caption{Numerical results of Example \ref{diab}}
    \begin{tabular}{|c|c|c|c|c|}\hline
    Method&BIC&ARI&errorRate&time(s) \\\hline
    em\_km&-5363.06&0.3371&0.289&\textbf{0.007} \\
    em\_mom&-5222.11&\textbf{0.6355}&\textbf{0.144}&0.380  \\
    em\_mbhc&\textbf{-5221.32}&\textbf{0.6355}&\textbf{0.144}&0.008  \\
    em\_emEM&-5221.33&0.6207&0.151&0.049 \\\hline
    \end{tabular}%
  \label{table8}
\end{table}%
Despite the fact that k-means method is the fastest method in this example, the ARI and the BIC are noticeably lower than in the other methods. Concerning the method of moments, it succeeds to have quite similar scores to the other methods in this example, but with a bigger computation time.   
\end{example}
\begin{example}[Olive oil]\label{olive}
The olive oil data set contains the chemical composition (8 chemical properties) of 572 olive oils. They are derived from three different macro-areas in Italy (South, Sardinia and Centre North). The dataset contains nine regions from
which the olive oils were taken in Italy. Thus we can cluster this dataset according to the macro-areas ($r=3$) or the region ($r=9$). As the number of features in this dataset is $m=8$, we choose $r=3$, so that the condition $r\leq m$ for the method of moment is verified.\\
\begin{table}[ht]
  \centering
  \caption{Numerical results of Example \ref{olive}}
    \begin{tabular}{|c|c|c|c|c|}\hline
    Method&BIC&ARI&errorRate&time(s) \\\hline
    em\_km&-10948.64 & 0.4018 &0.262 &\textbf{0.021} \\
    em\_mom&-10946.46 &0.4532 &0.210 &0.508 \\
    em\_mbhc&\textbf{-10625.59} &\textbf{0.5003} & \textbf{0.185} &0.080  \\
    em\_emEM&-10948.72 & 0.4040 &0.260 &0.087 \\\hline
    \end{tabular}%
  \label{table9}
\end{table}%
The results show that the MBHC initialisation strategy yields the largest BIC, the highest ARI and the lowest errorRate values among the other initialisation strategies. Nevertheless, the initialisation by the moment method comes in second position after the MBHC strategy in terms of the BIC, ARI and errorRate values, while the K-means and the emEM initialisation strategies attain almost the same values of the previously mentionned criteria.
\end{example}
This shows that for these datasets which are not well fitted by the mixture of spherical Gaussians, the moment method can still give good initialisations for the EM algorithm, in comparision with the common initialisation strategies.
\begin{example}[MNIST digit image database]\label{mnist} 
The MNIST digit image database \cite{deng2012mnist} is a large database that contains images of $28\times28$ pixels for handwritten digits (0 to 9). Each pixel contains an integer between 0 and 255 that represents the grayscale levels. The number of features is $28\times28=784$. We choose the MNIST digit image dataset which contains 60000 images. We take a subset of this dataset that contains the images of label 0 or 1. The size of the subset is 12665 images. Since the number of features is quite large (784), and we aim to test a spherical Gaussian mixture model, a good practice in this case is to apply one of the dimensionality reduction strategies. Roughly speaking, the dimensionality reduction strategies aim to reduce the number of features such that a high percentage of the information within the dataset is conserved. In other words, the performance in term of accuracy of the clustering methods will not be noticeably affected by this reduction, and on the other hand this will reduce considerably the time of computation. For this purpose, we choose to apply the Principal Component Analysis transformation (PCA) \cite{doi:10.1080/14786440109462720,Jolliffe2011}. We conserve the first five variables given by this transformation (see \Cref{fig3}). The dataset that we consider in this example contains 12665 observations, the number of clusters is $r=2$, and the number of features is $m=5$. We apply the different initialisation strategies and we report the results in \Cref{table10}.\\
\begin{figure}[ht]
\centerline{\includegraphics[width=5in, height=5in]{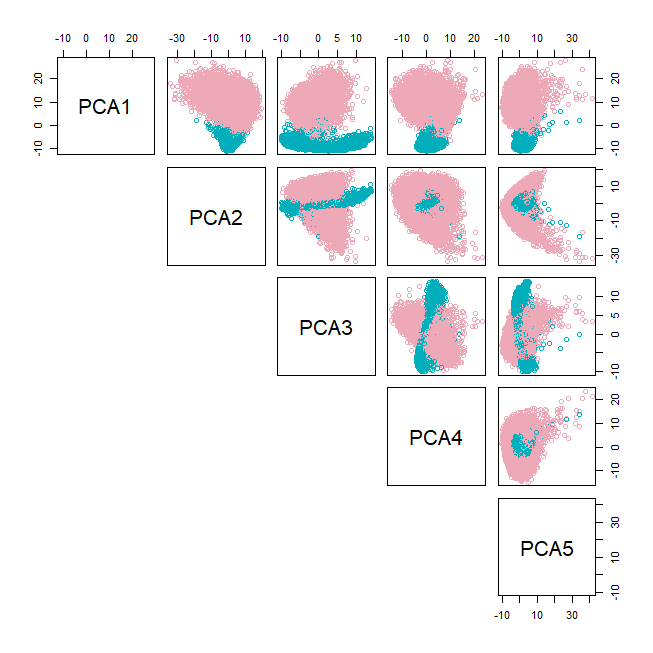}}
\caption{Scatterplot for pairs of variables: upper panels show the first five features obtained by applying the PCA transformation on the dataset of \Cref{mnist}. The graphs points marked according to the true two classes 0 and 1.}
\label{fig3}
\end{figure}
\begin{table}[ht]
  \centering
  \caption{Numerical results of Example \ref{olive}}
    \begin{tabular}{|c|c|c|c|c|}\hline
    Method&BIC&ARI&errorRate&time(s) \\\hline
    em\_km& -384977.3&0.9304&\textbf{0.017}&\textbf{0.537} \\
    em\_mom&-384978.2& \textbf{0.9308}& \textbf{0.017}& 1.87 \\
    em\_mbhc& \textbf{-382746.2}& 0.2445&0.252&543.4  \\
    em\_emEM&-384977.6&0.9301& 0.0177655&1.80 \\\hline
    \end{tabular}%
  \label{table10}
\end{table}%
As we can see, the results given by the method of moments in \Cref{table10} are very satisfactory in comparison with the other initialisation strategies with ARI$=0.9308$. In particular, the method of moments clearly outperforms MBHC method in this regard, in term of accuracy and the time of computation. In fact, the MBHC takes 543.4 seconds without reaching a \emph{good} ARI score. This example sheds some light on the performance of the method of moments. The large number of samples (in this example equal to 12665) does not have a high impact on the computation time, which is not the case, for the MBHC method, where this factor increases significantly its computation time. Moreover, it is true that a large number of features could have a negative impact on the computation time of the method of moments, but it is not a sever limitation since as we saw in this example, this can be efficiently remedied by applying one of the dimensionality reduction techniques. In this regard, some recent work \cite{pereira2022tensor} studies how the computation complexity of the moment method can be reduced while conserving its desirable high accuracy property. Conducting more research in this direction, we believe that the method of moments will have more sophisticated and competitive (in term of computation time) developments in the future. 
\end{example}

\section{Conclusion}

In the context of unsupervised machine learning, the type of models to
be recovered plays an important role. For Gaussian mixture models,
where iterative methods such as Expectation Maximisation algorithms
are applied, the choice of the initialisation is also crucial to
recover an accurate model of a given dataset. We demonstrated in the
experimentation that tensor decomposition techniques can provide a
good initial point for the EM algorithm, and that the moment tensor
method outperforms the other state-of-the-art strategies, when
datasets are well represented by spherical Gaussian mixture
models. For that purpose, we presented a new tensor decomposition
algorithm adapted to the decomposition of identifiable tensors with
low interpolation degree, which applies to a 3$^{\mathrm rd}$ order
moment tensors associated to the data distribution as we have shown.

\paragraph{\bf Acknowledgement}
We would like to thank the anonymous reviewers for their valuable remarks that helped us improving the paper. 
\noindent{}\small
\newcommand{\etalchar}[1]{$^{#1}$}

\end{document}

%% file: my-notation.tex
\theoremstyle{definition}
\newtheorem{definition}{Definition}[section]

\newtheorem{example}[definition]{Example}

\theoremstyle{plain}
\newtheorem{lemma}[definition]{Lemma}
\newtheorem{proposition}[definition]{Proposition}
\newtheorem{theorem}[definition]{Theorem}

\newtheorem{assumption}[definition]{Assumption}

\def\CC{\mathbb{C}}
\def\EE{\mathbb{E}}
\def\NN{\mathbb{N}}
\def\RR{\mathbb{R}}

\newcommand{\vb}[1]{\mathbf{#1}}
\newcommand{\apolar}[1]{\langle{#1}\rangle}
\newcommand{\rank}{\mathop{\mathrm{rank}}}

\newcommand{\im}{\mathop{\mathrm{im}}}
\newcommand{\diag}{\mathop{\mathrm{diag}}}
\newcommand{\strong}[1]{\textbf{#1}}

\def\Gl{\mathrm{Gl}}